\documentclass[12pt]{amsproc}
\newtheorem{theorem}{\sc Theorem}[section]
\newtheorem{lemma}[theorem]{\sc Lemma}
\newtheorem{proposition}[theorem]{\sc Proposition}
\newtheorem{corollary}[theorem]{\sc Corollary}

\begin{document}
\title{Linear groups with almost right Engel elements}
\author{Pavel Shumyatsky }
\address{ Department of Mathematics, University of Brasilia,
Brasilia-DF, 70910-900 Brazil}
\email{pavel@unb.br}
\thanks{This research was supported by FAPDF and CNPq-Brazil}
\keywords{Linear groups; Engel condition; locally nilpotent groups}
\subjclass[2010]{20E99, 20F45, 20H20}
\begin{abstract}
Let $G$ be a linear group such that for every $g\in G$ there is a finite set ${\mathcal R}(g)$ with the property that for every $x\in G$ all sufficiently long commutators $[g,x,x,\dots,x]$ belong to ${\mathcal R}(g)$. We prove that $G$ is finite-by-hypercentral.
\end{abstract}

\maketitle
\section{Introduction}

By a linear group we understand a subgroup of $GL(m,F)$ for some field $F$ and a positive integer $m$. A group $G$ is called an \emph{Engel group} if for every $x,g\in G$ the equation $[x,g,g,\dots , g]=1$ holds, where $g$ is repeated in the commutator sufficiently many times depending on $x$ and $g$. (Throughout the paper, we use the left-normed simple commutator notation
$[a_1,a_2,a_3,\dots ,a_r]=[...[[a_1,a_2],a_3],\dots ,a_r]$.)
Of course, any locally nilpotent group is an Engel group. In some classes of groups the converse is also known to be true. For example,  a finite Engel group is nilpotent by Zorn's theorem \cite{zorn}. Gruenberg proved that linear Engel groups are locally nilpotent \cite{grue}. Wilson and Zelmanov~\cite{wi-ze} proved that profinite Engel groups are locally nilpotent, and Medvedev~\cite{med} extended this result to compact (Hausdorff) groups.

A left Engel sink of an element $g\in  G$ is a set ${\mathcal L}(g)$ such that for every $x\in G$ all sufficiently long commutators $[x,g,g,\dots ,g]$ belong to ${\mathcal L}(g)$, that is, for every $x\in G$ there is a positive integer $n(x,g)$ such that
 $$[x,\underbrace{g,g,\dots ,g}_n]\in {\mathcal L}(g)\qquad \text{for all }n\geq n(x,g). $$ This notation is ambiguous in the sense that ${\mathcal L}(g)$ is not unique. Engel groups are precisely the groups for which we can choose ${\mathcal L}(g)=\{ 1\}$ for all $g\in G$. We say that an element $g\in G$ is almost left Engel if there is a finite left Engel sink ${\mathcal L}(g)$. For such elements the sink ${\mathcal L}(g)$ can be chosen to be minimal, thus eliminating ambiguity in this notation.

In \cite{monat} and \cite{khushu} we considered linear, profinite, and compact groups all of whose elements are almost left Engel. It was shown that all such groups are finite-by-(locally nilpotent).

Define a right Engel sink of $g\in G$ as a set ${\mathcal R}(g)$ such that 
for every $x\in G$ all sufficiently long commutators $[g,x,x,\dots,x]$ belong to ${\mathcal R}(g)$. Say that the element $g$ is almost right Engel if it has a finite right Engel sink.

The purpose of the present article is to establish the following result.

\begin{theorem}\label{main} Suppose that all elements in a linear group $G$ are almost right Engel. Then $G$ is finite-by-hypercentral.
\end{theorem}

Recall that the union of all terms of the (transfinite) upper central series of $G$ is called the hypercenter. The group $G$ is hypercentral if it coincides with its hypercenter. The hypercentral groups are known to be locally nilpotent (see \cite[12.2.4]{rob}). By well-known theorems obtained by Garascuk \cite{gara} and Gruenberg \cite{grue}, Engel linear groups are hypercentral. Theorem \ref{main} provides a natural generalization of this result. It can also be viewed as a dual for the result that linear groups with almost left Engel elements are finite-by-hypercentral \cite{monat}.

\section{Right Engel sinks in finite groups} Throughout the article, we write ${\mathcal R}(g)$ for the minimal right Engel sink of an  almost right Engel element $g$ in a group $G$. We write ${\mathcal L}(g)$ for the minimal left Engel sink of  an almost  left Engel element $g\in G$. The commutator $[x,y,y,\dots,y]$, where $y$ is repeated $n$ times, will be denoted by $[x,{}_n\,y]$. A coprime automorphism of a finite group $G$ is an automorphism of order relatively prime to the order of $G$. If $A$ is a group of automorphisms of $G$, the symbol $[G,a]$ denotes the subgroup of $G$ generated by all elements of the form $g^{-1}g^a$, where $g\in G$ and $a\in A$. The set of prime divisors of order of a finite group $G$ is denoted by $\pi(G)$. We write $\langle X\rangle$ to denote the subgroup generated by a set $X$. The expression ``$(a,b,\dots )$-bounded'' means ``bounded from above in terms of  $a,b,\dots$ only''.

Throughout, we use the following standard facts on coprime groups of automorphisms without explicit references (\cite[Theorem 6.2.2, Theorem 5.3.5, Theorem 5.3.6]{go}).
\begin{lemma}\label{11} Let $A$ be a group of automorphisms of the finite group $G$ with $(|A|,|G|)=1$.   
\begin{enumerate}
\item If $N$ is any $A$-invariant normal subgroup of $G$, we have $C_{G/N}(A)=C_G(A)N/N$;
\item $G=[G,A]C_G(A)$;
\item If $G$ is abelian, then $G=[G,A]\times C_G(A)$;
\item $[[G,A],A]=[G,A]$;
\item $A$ leaves invariant a Sylow $p$-subgroup of $G$ for each prime $p\in\pi(G)$.
\end{enumerate}
\end{lemma}

\begin{lemma}\label{s11} Let $G=H\langle a\rangle$ be a finite group where $H$ is a normal abelian subgroup and $(|a|,|H|)=1$. Then ${\mathcal R}(a)=[H,a]$.
\end{lemma}
\begin{proof} Choose $x\in H$. Since $H$ is abelian, we have $[a,{}_n\,(ax)]=[x^{-1},{}_{n+1}\,a]$ and the result follows.
\end{proof}

In the next lemma we use the well-known fact that any right Engel element in a finite group $G$ belongs to the hypercenter of $G$.
\begin{lemma}\label{nilposs} Let $G=H\langle a\rangle$ be a finite group where $H$ is a normal nilpotent subgroup and $(|a|,|H|)=1$. Then $|[H,a]|$ is $\vert{\mathcal R}(a)\vert$-bounded.
\end{lemma}
\begin{proof} Without loss of generality we can assume that $H=[H,a]$. Suppose that $N$ is a normal $a$-invariant subgroup of $H$ such that ${\mathcal R}(a)\cap N=1$. Then $a$ belongs to the hypercenter of $N\langle a\rangle$ and hence $[N,a]=1$. Taking into account that $H=[H,a]$ we deduce that $N\leq Z(H)$. In particular, it follows that ${\mathcal R}(a)\cap Z_2(G)\not=1$. Arguing by induction on $\vert{\mathcal R}(a)\vert$ and passing to the quotient $G/Z_2(G)$ we deduce that the nilpotency class of $H$ is at most 2$\vert{\mathcal R}(a)\vert$. Therefore we can use induction on the nilpotency class of $H$. By induction $|H/Z(H)|$ is $\vert{\mathcal R}(a)\vert$-bounded. In view of Schur's theorem \cite[10.1.4]{rob} we conclude that $H'$ has $\vert{\mathcal R}(a)\vert$-bounded order. Pass to the quotient $G/H'$ and assume that $H$ is abelian. Now Lemma \ref{s11} states that $H={\mathcal R}(a)$ and the result follows.
\end{proof}

\begin{lemma}\label{golds} Let $G$ be a finite group acted on by a finite group $A$ such that $(|A|,|G|)=1$. Then $[G,A]$ is generated by all nilpotent subgroups $T$ such that $T=[T,A]$.
\end{lemma}
\begin{proof} We can assume that $G=[G,A]$. Let $K$ be the subgroup of $G$ generated by all nilpotent subgroups $T$ such that $T=[T,A]$. We need to show that $K=G$. Obviously, $C_G(A)\leq N_G(K)$. If $P$ is any nilpotent $A$-invariant subgroup, we have $P=[P,A]C_P(A)$. Here $[P,A]\leq K$ and $C_P(A)\leq N_G(K)$. We see that each nilpotent $A$-invariant subgroup normalizes $K$. Since $G$ has an $A$-invariant Sylow $p$-subgroup for each prime $p\in\pi(G)$ and since $G$ is generated by such Sylow $p$-subgroups, we conclude that $K$ is normal in $G$. We observe that $A$ acts trivially on nilpotent $A$-invariant subgroup of $G/K$. Therefore $A$ acts trivially on $G/K$. Since $G=[G,A]$, it follows that $G/K=[G/K,A]$ and so $G/K=1$, as required.
\end{proof}
Remark that the above lemma for soluble groups $G$ was established by Goldschmidt in \cite[Lemma 2.1]{gold}. 

\begin{lemma}\label{s12} Let $G=H\langle a\rangle$ be a finite group where $H$ is a normal subgroup and $(|a|,|H|)=1$. Then $\langle{\mathcal R}(a)\rangle=[H,a]$.
\end{lemma}

\begin{proof} Since $a$ centralizes $H/[H,a]$, it is clear that $\langle{\mathcal R}(a)\rangle\leq[H,a]$. Therefore it is sufficient to show that $H=\langle{\mathcal R}(a)\rangle$ under the additional assumption that $H=[H,a]$. Lemma \ref{golds} shows that $H$ is generated by its $a$-invariant $p$-subgroups $T$ such that $T=[T,a]$. Therefore, without loss of generality we can assume that $H$ is a $p$-group for some prime $p$. Then it is sufficient to show that $H$ is generated by ${\mathcal R}(\alpha)$ modulo the Frattini subgroup $\Phi(H)$. We can pass to the quotient $G/\Phi(H)$ and assume that $H$ is abelian. Now the lemma is immediate from Lemma \ref{s11}.
\end{proof}
\begin{lemma}\label{ssss12} Let $G=H\langle a\rangle$ be a finite group where $H$ is a normal subgroup and $(|a|,|H|)=1$. Then $|[H,a]|$ is $\vert{\mathcal R}(a)\vert$-bounded.
\end{lemma}
\begin{proof} Without loss of generality assume that $H=[H,a]$. By Lemma \ref{s12} we have $\langle{\mathcal R}(a)\rangle=H$. Therefore $C_H({\mathcal R}(a))\leq Z(H)$. Since $C_H(a)$ normalizes ${\mathcal R}(a)$, the index of $C_H(a)\cap Z(H)$ in $C_H(a)$ is bounded in terms of $|{\mathcal R}(a)|$ only. Let $p\in\pi(H)$ and $P$ be an $a$-invariant Sylow $p$-subgroup of $H$. We have $P=[P,a]C_P(a)$. Lemma \ref{nilposs} states that $|[P,a]|$ is $\vert{\mathcal R}(a)\vert$-bounded. Hence, the index of $P\cap Z(H)$ in $P$ is bounded in terms of $|{\mathcal R}(a)|$ only. Since $H$ has an $a$-invariant Sylow $p$-subgroup for every $p\in\pi(H)$, we deduce that the index of $Z(H)$ in $H$ is bounded in terms of $|{\mathcal R}(a)|$ only. Schur's theorem now says that the order of the derived group $H'$ is bounded. We can pass to the quotient $G/H'$ and assume that $H$ is abelian. Now the result is immediate from Lemma \ref{s11}.
\end{proof}
\section{Almost right Engel elements in infinite groups}
In the present section we will establish two useful results about almost right Engel elements in infinite groups (see Propositions \ref{hyperce} and Proposition \ref{ssss}). If $a$ is an almost right Engel element in a group $G$ and $K$ is a subgroup of $G$, we write ${\mathcal R}_K(a)$ to denote the minimal subset of $G$ containing all sufficiently long commutators $[a,x,x,\dots,x]$ where $x\in K$. If $a$ is an almost left Engel element in $G$, we write ${\mathcal L}_K(a)$ to denote the minimal subset of $G$ containing all sufficiently long commutators $[x,a,a,\dots,a]$ where $x\in K$.

\begin{lemma} \label{nach} Let $G=H\langle a\rangle$ be a group with a normal abelian subgroup $H$. Then $a$ is an almost right Engel element if and only if it is an almost left Engel element. In that case ${\mathcal L}(a)={\mathcal R}(a)$.
\end{lemma}
\begin{proof} The lemma follows from the fact that if $x\in H$, for each $n$ we have $[a,{}_n\,(ax)]=[x^{-1},{}_{n+1}\,a]$. It is easy to see that the intersection of all normal subgroups of $G$ whose quotients are locally nilpotent is ${\mathcal L}(a)$ as well as ${\mathcal R}(a)$.
\end{proof}
A well-known theorem of Baer states that if, for a group $G$ and a positive integer $k$, the quotient $G/Z_k(G)$ is finite, then so is $\gamma_{k+1}(G)$ (see \cite[14.5.1] {rob}). Here $\gamma_{k+1}(G)$ denotes the term of the lower central series of $G$. Recently, the following related result was obtained in \cite{degiova} (see also \cite{kurda}).

\begin{theorem}\label{sysak} Let $G$ be a group and let $H$ be the hypercenter of $G$. If $G/H$ is finite, then $G$ has a finite normal subgroup $N$ such that $G/N$ is hypercentral.
\end{theorem}

\begin{proposition}\label{hyperce} Let $G=H\langle a\rangle$ where $H$ is a normal hypercentral subgroup and $a$ is an almost right Engel element. Then $G$ is finite-by-hypercentral.
\end{proposition}
\begin{proof} If ${\mathcal R}(a)=1$, then $a$ is a right Engel element. In that case $a^{-1}$ is a left Engel element (see \cite[12.3.1] {rob}) and, by \cite[Lemma 3.5]{monat}, the group $G$ is hypercentral. We therefore will assume that ${\mathcal R}(a)\neq1$ and use induction on the number of elements in ${\mathcal R}(a)$.

Suppose first that Z(G)=1. Set $Z=Z(H)$ and $K=Z\langle a\rangle$. By Lemma \ref{nach}, Then ${\mathcal L}_K(a)={\mathcal R}_K(a)$. Since $Z(G)=1$, it follows that ${\mathcal L}_K(a)\neq1$. Moreover, by \cite[Lemma 3.6]{monat}, the subgroup $\langle{\mathcal L}_K(a)\rangle$ is finite. Since ${\mathcal L}_K(a)$ is contained in $Z(H)$, we observe that $\langle{\mathcal L}_K(a)\rangle$ is normal in $G$. Thus, we can pass to the quotient-group $G/\langle{\mathcal L}_K(a)\rangle$. By induction on $|{\mathcal R}(a)|$ we deduce that $G$ is finite-by-hypercentral.

Now drop the assumption that $Z(G)=1$. Let $T$ be the hypercenter of $G$. The above paragraph shows that $G/T$ is finite-by-hypercentral. If $G/T$ is hypercentral, then $G$ itself is hypercentral so without loss of generality we can assume that $G/T$ has a nontrivial finite subgroup $N/T$ such that $G/N$ is hypercentral and ${\mathcal R}(a)\leq N$. Let $m$ be a positive integer such that $a^m$ centralizes $N/T$. Then $T\langle a^m\rangle$ is contained in the hypercenter of $N\langle a\rangle$. Therefore $N\langle a\rangle$ is hypercentral-by-finite. In view of Theorem \ref{sysak} $N\langle a\rangle$ is finite-by-hypercentral and so $N\langle a\rangle$ has a characteristic finite subgroup $M$ such that $N\langle a\rangle/M$ is hypercentral. Suppose that ${\mathcal R}(a)\cap M=1$. Then $a$ is right Engel in $N\langle a\rangle$. We see that $a^{-1}$ is left Engel in $G$, whence by \cite[Lemma 3.5]{monat} the group $G$ is hypercentral. Therefore without loss of generality we can assume that $M$ is nontrivial and ${\mathcal R}(a)\cap M\neq 1$. We pass to the quotient $G/M$ and by induction on $|{\mathcal R}(a)|$ deduce that $G$ is finite-by-hypercentral.
\end{proof} 

Groups in which every finitely generated subgroup is finite are called locally finite. In general, there is no good Sylow theory for locally finite groups. However, it is easy to see that if, for some prime $p$, a locally finite group $G$ has at least one finite maximal $p$-subgroup, then all $p$-subgroups of $G$ are finite and the maximal ones are conjugate. We therefore can talk about locally finite groups with finite Sylow $p$-subgroups. In what follows we denote by $O(G)$ the maximal normal subgroup of $G$ that consists of elements of odd order. Of course, the Feit-Thompson theorem \cite{fetho} implies that a locally finite group in which all elements have odd order is locally soluble. We say that a group is virtually locally soluble if it has a locally soluble subgroup of finite index. The following lemma is immediate from \cite[Theorem 3.17]{kewe}.
\begin{lemma} \label{ooo} Let $G$ be a locally finite group whose Sylow $2$-subgroups are finite. If $G$ is virtually locally soluble, then $O(G)$ has finite index in $G$.
\end{lemma}
\begin{proposition} \label{ssss} Let $G$ be a locally finite group whose Sylow $2$-subgroups are finite and suppose that ${\mathcal R}(g)$ is finite for each $2$-element $g\in G$. Then $G$ has a finite normal subgroup $N$ such that $G/N$ consists of elements of odd order. In particular $G$ is virtually locally soluble.
\end{proposition}

\begin{proof} Choose a Sylow 2-subgroup $S$ in $G$. We will use induction on $|S|$. Since the case $S=1$ is self-evident, assume that $S\neq1$. By induction, the result holds for $C_G(x)/\langle x\rangle$ whenever $x$ is a nontrivial 2-element. It follows that $C_G(x)$ is virtually locally soluble. Choose a nontrivial element $x\in S$. Let $C=C_G(x)$, and let $D$ be a Sylow 2-subgroup of $C$. By Lemma \ref{ooo}, $O(C)$ has finite index in $C$. Lemma \ref{ssss12} shows that $[O(C),D]$ is finite. It follows that $C$ has only finitely many Sylow 2-subgroups and hence only finitely many 2-elements. Thus, an arbitrary nontrivial element $x\in S$ centralizes only finitely many 2-elements. 

Choose an involution $a\in S$ and denote by $Y$ the set of 2-elements $y\in G$ such that $y^a=y^{-1}$. We already know that $Y$ contains only finitely many involutions. We will now show that actually the set $Y$ is finite. Suppose that $Y$ is infinite. Let $k$ be the maximum number such that $Y$ contains only finitely many elements of order $k$. Such $k$ exists because the Sylow 2-subgroups in $G$ are finite and therefore of finite exponent. Since $Y$ is infinite, there is an element $b\in Y$ of order $k$ and infinitely many elements $u\in Y$ such that $u^2=b$. This is a contradiction since $C_G(b)$ contains only finitely many 2-elements. Thus, we have shown that $Y$ is finite.

Remark that $[g,a]^a=[g,a]^{-1}$ for any $g\in G$. By Lemma \ref{s11}, any odd order element $h$ such that $h^a=h^{-1}$ lies in ${\mathcal R}(a)$. Therefore every commutator of the form $[g,a]$ can be written as a product $[g,a]=ry$ where $r\in{\mathcal R}(a)$ and $y\in Y$. Since both $Y$ and ${\mathcal R}(a)$ are finite, there are only finitely many such commutators and therefore $C_G(a)$ has finite index in $G$. Dicman's lemma \cite[14.5.7]{rob} now states that $\langle a^G\rangle$ is finite. We pass to the quotient $G/\langle a^G\rangle$ and apply the induction hypothesis.
\end{proof}
\section{Proof of Theorem \ref{main}}

Linear groups are naturally equipped with the Zariski topology. If $G$ is a linear group, the connected component of $G$ containing 1 is denoted by $G^0$. We will use (sometimes implicitly) the following facts on linear groups. All these facts are well-known and are provided here just for the reader's convenience.

\begin{itemize}
\item If $G$ is a linear group and $N$ a normal subgroup which is closed in the Zarisky topology, then $G/N$ is linear (see \cite[Theorem 6.4]{wehr}).

\item Since finite subsets of $G$ are closed in the Zariski topology, it follows that any finite subgroup of a linear group is closed. Hence $G/N$ is linear for any finite normal subgroup $N$.

\item If $G$ is a linear group, the connected component $G^0$ has finite index in $G$ (see \cite[Lemma 5.3]{wehr}).

\item Each finite conjugacy class in a linear group centralizes $G^0$ (see \cite[Lemma 5.5]{wehr}).

\item In a linear group any descending chain of centralizers is finite. 

\item A linear group generated by normal nilpotent subgroups is nilpotent (see Gruenberg \cite{grue}).

\item Tits alternative: A finitely generated linear group either is virtually soluble or contains a subgroup isomorphic to a nonabelian free group (see \cite{tits}).

\item The Burnside-Schur theorem: A periodic linear group is locally finite (see \cite[9.1]{wehr}).

\item Zassenhaus theorem: A locally soluble linear group is soluble. Every linear group contains a unique maximal soluble normal subgroup (see \cite[Corollary 3.8]{wehr}).

\item Since the closure in the Zariski topology of a soluble subgroup is again soluble (see \cite[Lemma 5.11]{wehr}), it follows that the unique maximal soluble normal subgroup of a linear group is closed. In particular, if $G$ is linear and $R$ is the unique maximal soluble normal subgroup of $G$, then $G/R$ is linear and has no nontrivial normal soluble subgroups.

\item Gruenberg theorem: The set of left Engel elements in a linear group $G$ coincides with the Hirsch-Plotkin radical of $G$. The set of right Engel elements coincides with the hypercenter of $G$ (see \cite{grue}).
\end{itemize}

\noindent Here, as usual, the Hirsch-Plotkin radical of a group is the maximal normal locally nilpotent subgroup.

It will be convenient first to prove the theorem in the particular case where the group is virtually soluble.

\begin{lemma}\label{nilpott} A virtually soluble linear all of whose elements are almost right Engel is finite-by-hypercentral.
\end{lemma}
\begin{proof} Suppose that $G$ is a virtually soluble linear group all of whose elements are almost right Engel. Let $S$ be a normal soluble subgroup of finite index in $G$. By induction on the derived length of $S$ we assume that $S'$ is finite-by-hypercentral. Therefore $S'$ has a finite characteristic subgroup $M$ such that $S'/M$ is hypercentral. Passing to the quotient $G/M$ without loss of generality we can assume that $S'$ is hypercentral. By Proposition \ref{hyperce} the subgroup $\langle S',x\rangle$ is finite-by-hypercentral for each $x\in G$. Thus, for each $x\in G$ there exists a finite characteristic subgroup $R_x\leq\langle S',x\rangle$ such that $\langle S',x\rangle/R_x$ is hypercentral. Since $\langle S',x\rangle$ is normal in $S$ whenever $x\in S$, it follows that each element in $R_x$ has centralizer of finite index in $S$, hence centralizer of finite index in $G$. Therefore $G^0$ centralizes $R_x$. It follows that whenever $x\in G^0$, the element $x$ is a right Engel element in $\langle S',x\rangle$. In that case $x^{-1}$ is a left Engel element (see \cite[12.3.1] {rob}) and, by \cite[Lemma 3.5]{monat}, the group $\langle S',x\rangle$ is hypercentral. Further, we observe that the subgroup $\prod\langle S',x\rangle$, where $x$ ranges over $S\cap G^0$, is locally nilpotent and therefore hypercentral because linear locally nilpotent groups are hypercentral. In particular, $N=S\cap G^0$ is hypercentral and so $G$ is virtually hypercentral. By Proposition \ref{hyperce} the subgroup $\langle N,x\rangle$ is finite-by-hypercentral for each $x\in G$. In other words, for each $x\in G$ there exists a finite characteristic subgroup $Q_x\leq\langle N,x\rangle$ such that the quotient $\langle N,x\rangle/Q_x$ is hypercentral. Since $N$ has finite index in $G$, it follows that $G$ contains only finitely many subgroups of the form $\langle N,x\rangle$. Set $N_0=\prod_{x\in G}Q_x$. We see that $N_0$ is a finite normal subgroup. Pass to the quotient $G/N_0$. Now the subgroup $\langle N,x\rangle$ is hypercentral for each $x\in G$. It follows that $N$ consists of right Engel elements and so, by a result of Gruenberg, $N$ is contained in the hypercenter of $G$. Theorem \ref{sysak} now states that $G$ is finite-by-hypercentral, as required. 
\end{proof}

In general, an extension of a locally soluble group by another locally soluble group need not be locally soluble. However for linear groups we have the following lemma.

\begin{lemma}\label{seri} Let $G$ be a linear group having a finite normal series all of whose quotients are virtually locally soluble. Then $G$ is virtually soluble.
\end{lemma}
\begin{proof} Taking into account the Zassenhaus theorem, this is just an easy induction on the length of the normal series of $G$. 
\end{proof}

We are now ready to prove Theorem \ref{main} in its full generality. For the reader's convenience we restate it here.

\begin{theorem} Let $G$ be a linear group all of whose elements are almost right Engel. Then $G$ is finite-by-hypercentral. 
\end{theorem}
\begin{proof} In view of Proposition \ref{hyperce} it is sufficient to show that $G$ is virtually soluble. By the Zassenhaus theorem a linear group is soluble if and only if it is locally soluble. Therefore it is sufficient to show that $G$ is virtually locally soluble. It is clear that $G$ does not contain a subgroup isomorphic to a nonabelian free group. Hence, by Tits alternative, any finitely generated subgroup of $G$ is virtually soluble. Therefore, by Proposition \ref{hyperce}, any finitely generated subgroup of $G$ is finite-by-hypercentral. Since hypercentral groups are locally nilpotent, it becomes obvious that elements of finite order in $G$ generate a periodic subgroup. Moreover, the quotient of $G$ over the subgroup generated by all elements of finite order is locally nilpotent. Hence, in view of Lemma \ref{seri} the group $G$ is virtually locally soluble if and only if so is the subgroup generated by elements of finite order. Therefore without loss of generality we can assume that $G$ is an infinite periodic group. In view of the Burnside-Schur Theorem $G$ is locally finite. 

Let $R$ be the soluble radical of $G$. This is closed in the Zariski topology and therefore $G/R$ is linear. We can pass to the quotient and without loss of generality assume that $R=1$. In particular, by a result of Gruenberg, $G$ has no nontrivial left Engel elements. Taking into account Proposition \ref{ssss} we deduce that $G$ contains an infinite $2$-subgroup. The latter contains an infinite abelian subgroup $A$ \cite[Exercise 14.4.4]{rob}.  Since $G$ satisfies the minimal condition on centralizers, it follows that $G$ has a subgroup $D$ such that the centralizer $C=C_A(D)$ is infinite while $C_A(\langle D,x\rangle)$ is finite for each $x\in G\setminus D$. Choose an involution $a\in C$. The centralizer $C$ normalizes the finite set $\mathcal R(a)$. Hence, $C$ contains a subgroup of finite index which centralizes $\mathcal R(a)$. Because of minimality it follows that $C$ centralizes $\langle\mathcal R(a),D\rangle$. Since $C$ is infinite, we conclude that $\mathcal R(a)\subseteq D$.  In particular, $a$ centralizes $\mathcal R(a)$. Since $a$ is not an Engel element, the Baer-Suzuki theorem \cite[Theorem 3.8.2]{go} guarantees that there exists an element $y\in G$ of odd order such that $y^a=y^{-1}$. Obviously $\langle y\rangle=[\langle y\rangle,a]$. Thus, Lemma \ref{s11} shows that $y\in\mathcal R(a)$ and $y\not\in C_G(a)$, a contradiction. This completes the proof. 
\end{proof}

The next result is an easy corollary of Theorem \ref{main}.

\begin{corollary}\label{ma} Suppose that all elements in a linear group $G$ are almost right Engel and for every $g\in G$ there is a positive integer $n$ such that $[g,{}_s\, x]\in{\mathcal R}(g)$ whenever $x\in G$ and $s\geq n$. Then $G$ is finite-by-nilpotent.
\end{corollary}

\begin{proof} Theorem \ref{main} states that $G$ is finite-by-hypercentral. Passing to a quotient over a finite normal subgroup we can assume that $G$ is hypercentral. Then obviously for each element $x$ of $G$ there is a positive integer $k$ depending on $x$ such that $x$ is right $k$-Engel. It follows that $x^{-1}$ is left $(k+1)$-Engel. By a result of Gruenberg \cite{grue}, $G$ is nilpotent.
\end{proof}

\end{document}